\documentclass{amsart}

\usepackage[utf8]{inputenc}
\usepackage{amsmath,amssymb}
\usepackage{tikz-cd}
\usepackage{amsthm}
\usepackage{mathtools}
\usepackage{enumitem}
\usepackage{bbold}
\usepackage{hyperref}

\newtheorem{theorem}[subsubsection]{Theorem}
\newtheorem{theorem*}{Theorem}
\newtheorem{lemma}[subsubsection]{Lemma}
\newtheorem{conjecture}[subsubsection]{Conjecture}

\newtheorem{proposition}[subsubsection]{Proposition}

\newcommand{\A}{A^{\bullet}}
\newcommand{\T}{T^{\bullet}}

\newcommand{\tor}{\rm tor}
\newcommand{\ad}{\rm ad}
\newcommand{\Std}{\textnormal{std}}
\newcommand{\Hodge}{\rm Hdg}
\newcommand{\Top}{\rm top}
\newcommand{\antidiag}{\mathsf{antidiag}}
\newcommand{\diag}{\mathsf{diag}}

\DeclareMathOperator{\Rep}{Rep}
\DeclareMathOperator{\Vect}{Vect}
\DeclareMathOperator{\sym}{Sym}

\title{Pushforward of Siegel flag varieties in the Chow ring}
\author{Simon Cooper}

\begin{document}
	\pagestyle{plain}
	
	\begin{abstract}
		Given a reductive group $G$ over an algebraically closed field and subsets $I,J \subset \Delta$ of the simple roots $\Delta$ determined by a choice of maximal torus and Borel subgroup, we obtain a closed embedding of flag varieties $L_{J}/L_{J}\cap P_{I} \hookrightarrow G/P_{I}$. In this paper we compute the class of the sub flag variety $[L_{J}/L_{J}\cap P_{I}] \in \A(G/P_{I})$ in the Chow ring for the ‘Siegel’ case where $G$ is a general symplectic group of rank $g$ and $P_{I}$ is the parabolic stabilising a maximal isotropic subspace. This corresponds, under the isomorphism with the tautological ring of the moduli space of principally polarised abelian varieties $\T(\mathcal{A}_{g}^{\tor}) \cong \A(G/P_{I})$, to the generator of the classes in the tautological ring which are supported on the toroidal boundary. We conjecture that this is not a coincidence.
	\end{abstract}
	\maketitle
	\section{Introduction}
	
	Let $G$ be a reductive group over a field $k$. Given subsets $I,J \subset \Delta$ of the simple roots (for some choice of maximal torus and Borel subgroup), there is an embedding $L_{J}/L_{J}\cap P_{I} \hookrightarrow G/P_{I}$. This gives a class in the Chow ring $[L_{J}/L_{J}\cap P_{I}] \in \A(G/P)$. Can we describe this class? In the case where $G=GSp_{2g}$ is the general symplectic group and $I \coloneqq \Delta \setminus \{2e_{g}\}$ and $J \coloneqq \Delta \setminus \{e_{1}-e_{2}\}$ are obtained from $\Delta$ by removing extremal vertices of the Dynkin diagram, we have the following answer in terms of the top Chern class of a homogeneous vector bundle:
	\begin{theorem*}[\ref{th-main}]
		\label{th-main-intro}
		Let $G=GSp_{2g}$, $I \coloneqq \Delta \setminus \{2e_{g}\}$ and $J \coloneqq \Delta \setminus \{e_{1}-e_{2}\}$. Then $$[L_{J}/L_{J}\cap P_{I}] = a_{J,g}.c_{g}(\mathcal{V}(\rho_{\Hodge})) \in \A(G/P_{I})$$ for some constant $a_{J,g} \in \mathbb{Q}\setminus \{0\}$, where $\mathcal{V}(\rho_{\Hodge})$ is the vector bundle corresponding to the irreducible subrepresentation of the dual of the standard representation $\rho_{\Hodge}\subset \Std^{\vee}|_{P_{I}}$. Moreover, $(-1)^{g}a_{J,g} > 0$. 
	\end{theorem*}
	Theorem \ref{th-main-intro} is motivated by, though essentially independent of, a connection between tautological rings of Shimura varieties and Chow rings of flag varieties (always with rational coefficients). The tautological ring of a Shimura variety $S$ is defined to be the subring of the Chow ring $\T(S)\subset \A(S)$ generated by Chern classes of automorphic vector bundles. The connection originates with a diagram of the following form (due to Milne \cite{Milne}).
	\[\begin{tikzcd}
		 \mathcal{Y} \arrow[d, "\pi"'] \arrow[r, "f"]  &  X^{\vee}\cong G/P\\ S^{\tor}
	\end{tikzcd}\] where $\pi$ is a $G$-torsor and $f$ is a smooth $G$-equivariant morphism. In positive characteristic it is known (due to Wedhorn-Ziegler \cite{Wedhorn-Ziegler} and conjectured to be true in characteristic zero too) that  
	$$\T(S^{\tor}) \cong \A(G/P)$$
	where $S^{\tor}$ is the special fibre of a smooth toroidal compactification of a Shimura variety of Hodge type and $G/P$ is the special fibre of an integral model of the compact dual of the symmetric domain defining the Shimura variety. 
	Building on this, in \cite{Cooper1} a conjectural description is given of the tautological ring of the Shimura variety itself in terms of Chow rings of flag varieties:
	\begin{conjecture}
		\label{conj-bdy}
    Let $S$ be the special fibre of an integral canonical model of a Shimura variety of Hodge type.
		$$\T(S) \cong \A(G/P)/([L_{J}/L_{J}\cap P] \mid J \in \mathcal{J} \text{ maximal w.r.t. }\leq)$$
		where $\mathcal{J} \coloneqq \{\Delta\setminus\{\alpha\} \mid \alpha\in\Delta \text{ and }P_{J}\subset G \text{ defined over }\mathbb{Q}\}$ admits a partial order $\leq$ according to the stratification of the boundary of a toroidal compactification $S^{\tor}$ of $S$.
	\end{conjecture}
	 This connection has primarily been used to transport information from the flag variety side to the tautological ring side. For example, by providing a presentation of the tautological ring using Borel's classic presentation of the cohomology ring of a flag variety (\cite{Borel}, see \ref{subsec-Chowrings}). Additionaly, in the case of the moduli space of principally polarised abelian varieties $S=\mathcal{A}_{g}$ there is a complete description of the tautological ring, due to van der Geer \cite{vanderGeer}: $$\T(\mathcal{A}_{g})\cong \A(G/P)/(\lambda_{g})$$ 
	where $\lambda_{g}$ is the top Chern class of the Hodge vector bundle $\Omega/\mathcal{A}_{g}$. In this case there is one maximal $J\in\mathcal{J}$ and the conjectural description \ref{conj-bdy} of the tautological ring predicts Theorem \ref{th-main-intro}, which is proved in this paper by independent means. Thus, Theorem \ref{th-main-intro} provides basic evidence for the conjecture.
	The statement of Theorem \ref{th-main-intro} does not reference tautological rings or Shimura varieties and can be considered purely as a question about Chow rings of flag varieties. Moreover, the proof that $[L_{J}/L_{J}\cap P]$ is proportional to $\lambda_{g}$ relies only on facts about Chow rings of flag varieties. The main idea is to use the self-intersection formula for the embedding $L_{J}/L_{J}\cap P \hookrightarrow G/P$ to prove that the class $[L_{J}/L_{J}\cap P]\in \A(G/P)$ vanishes when pulled back to $\A(L_{J}/L_{J}\cap P)$. This follows from the fact that the top Chern class of the normal bundle vanishes in this context (\ref{cor-chern-normal}). A description of the kernel of the pullback map $\A(G/P)\rightarrow \A(L_{J}/L_{J}\cap P)$ (\ref{le-pullback-map}) completes the computation, since the kernel is generated by $\lambda_{g}$. 

	\section*{Acknowledgements}
    The author would like to thank W. Goldring for his support and guidance and M. Brion for helpful commments on earlier drafts of this paper. This research was conducted whilst the author was supported by grant KAW 2018.0356.

	\section{Generalities}
	\subsection{Group theoretic input data}
	Let $G$ be a reductive group over an algebraically closed field $k$. A choice of maximal torus $T$ determines a root datum $(X^{*}(T),\Phi,X_{*}(T),\Phi^{\vee})$. Given a choice of Borel subgroup $B$ containing $T$, we denote a set of positive roots $\Phi^{+} \coloneqq \{\alpha \in \Phi^{+} \mid U_{-\alpha} \subset B\}$. This choice of positive roots means that ample line bundles on $G/B$ correspond to dominant and regular characters. Denote $\Delta$ the basis of simple roots. A subset $K \subset \Delta$ determines a standard parabolic subgroup $P_{K} \subset G$. Every parabolic subgroup is conjugate to a standard parabolic subgroup $B \subset P_{K} \subset G$ for some $K \subset \Delta$. For sake of exposition, we shall work mostly with standard parabolic subgroups. Denote $W = W(G,T) \coloneqq N_{G}(T)/T$ the Weyl group of $G$ with respect to $T$. Given a subset of the simple roots $I\subset \Delta$, denote $W_{I} \subset W$ the group generated by the simple reflections $\{s_{\alpha}\in W \mid \alpha \in I\}$. Denoting $L_{I}\coloneqq P_{I}/\mathcal{R}_{u}(P_{I})$ the Levi quotient of $P_{I}$, we have that $W(L_{I},T)\cong W_{I}$. 
	
	\subsection{Intersection theory}
	\label{subsec-Chowrings}
	Given a smooth scheme $X$, we denote $\A(X)$ the Chow ring with rational coefficients. If there is an action of $G$ on $X$ then we denote $\A_{G}(X)$ the equivariant Chow ring.
	
	Given a diagram of the form 
	\[\begin{tikzcd}
		\mathcal{Y} \arrow[d, "\pi"'] \arrow[r, "f"]  &  G/P_{I}\\ X
	\end{tikzcd}\] where $\pi$ is a $G$-torsor and $f$ is smooth and $G$-equivariant, pullback of vector bundles yields functors
	\[\begin{tikzcd}
		\Rep(P_{I}) \arrow[r, "\mathcal{V}"] \arrow[d, "\mathcal{W}"]& \Vect(G/P_{I}) \\ \Vect(X)
	\end{tikzcd}\]
	since there is an equivalence of categories $\Rep(P_{I}) \simeq \Vect_{G}(G/P_{I})$. The ‘inverse' functor sends $(E,\phi)$ to $P_{I} \rightarrow GL(E_{x})$ where $x = eP_{I} \in G/P_{I}$. On Chow rings we have 
	\[\begin{tikzcd}
		S^{W_{I}}\cong \A_{G}(G/P_{I}) \arrow[d] \arrow[r, "F_{0}"] & \A(G/P_{I}) \\ \A(X)
	\end{tikzcd}\] where $S \coloneqq \sym(X^{*}(T)_{\mathbb{Q}})$ and the Weyl group $W$ acts on $S$ via its action on $X^{*}(T)$.
	The characteristic homomorphism $F_{0}$ is surjective with kernel $(S^{W}_{+})$ \cite[6.7]{Brion},\cite{Borel} so 
	$$\A(G/P_{I}) \cong S^{W_{I}}/(S_{+}^{W})$$
	The input we require from intersection theory for the proof of Theorem \ref{th-main-intro} is the following classic result.
	\begin{proposition}[Self-intersection formula \cite{LascuMumfordScott}]
		Let $\iota \colon X \hookrightarrow Y$ be a closed embedding of smooth varieties over an algebraically closed field with normal bundle $\mathcal{N}$. Then $\iota^{*}\iota_{*}(\gamma) = c_{\Top}(\mathcal{N}).\gamma$ for all $\gamma \in \A(X)$.
	\end{proposition}

	\subsection{Standard sub flag varieties}
	
	 Let $I,J \subset \Delta$ and fix an embedding $L_{J} \hookrightarrow G$. Now, $L_{J} \cap P_{I}$ is a parabolic subgroup of $L_{J}$ and the closed embedding $L_{J} \hookrightarrow G$ descends to a closed embedding $\iota\colon L_{J}/L_{J}\cap P_{I} \hookrightarrow G/P_{I}$ (see for example \cite[1.8]{Jantzen}). The Levi subgroup $L_{J}$ has maximal torus $T$ and Borel subgroup $B \cap L_{J}$.  
	%\subsection{Homogeneous vector bundles on a flag variety}
	%Given a diagram of the form 
	%\[\begin{tikzcd}
	%	 & G/P \arrow[d] \\ X \arrow[r] & \left[*/P\right]
	%\end{tikzcd}\] where $P\subset G$ is a parabolic subgroup, $[*/P]$ denotes the classifying stack of $P$ and $X$ is a scheme then by considering pullback of vector bundles we obtain functors
	%\[\begin{tikzcd}
	%	\Rep(P) \arrow[r, "\mathcal{W}"] \arrow[d, "\mathcal{V}"]& \Vect(X) \\ \Vect(G/P)
	%\end{tikzcd}\]
	%The functor $\mathcal{V}$ induces an equivalence of categories $\Rep(P) \simeq \Vect_{G}(G/P)$. The ‘inverse' functor sends $(E,\phi)$ to $P \rightarrow GL(E_{x})$ where $x = eP \in G/P$.
	
	%\subsection{Borel's description of the Chow rings}
	%\label{subsec-Chowrings}
	%The Bruhat stratification gives an affine stratification of $G/P$ (a stratification by affine spaces, see e.g. \cite[13.8]{Jantzen})  which implies that the cycle class map $\A(G/P) \rightarrow \coH(G/P,\mathbb{Q})$ is an isomorphism. Borel \cite{Borel} gave a presentation for this ring as follows. Let $S \coloneqq \sym(X^{*}(T)_{\mathbb{Q}})$. Then $W$ acts on $S$ via its action on $T$ and 
	%$$H^{2\bullet}(G/P_{I},\mathbb{Q}) \cong S^{W_{P}}/(S_{+}^{W})$$

	\subsection{Pullback \& Pushforward}
	
	Suppose we have $\iota\colon L_{J}/L_{J}\cap P_{I} \hookrightarrow G/P_{I}$ as above. There is an inclusion of Weyl groups $W(L_{J},T) \cong W_{J} \subset W$. Pullback gives a map of Chow rings
	$$\iota^{*} \colon S^{W_{I}}/(S_{+}^W) \cong \A(G/P_{I}) \longrightarrow \A(L_{J}/L_{J}\cap P_{I}) \cong S^{W_{I\cap J}}/(S_{+}^{W_{J}})$$
	Moreover, if we consider the diagram
	\[\begin{tikzcd}
	L_{J}/L_{J} \cap B \arrow[r, hook, "\nu"] \arrow[d, "\theta"]& G/B \arrow[d, "\pi"]\\ L_{J}/L_{J} \cap P_{I} \arrow[r, hook, "\iota"]& G/P_{I}
	\end{tikzcd}\] 
	with the full flag varieties then pullback gives us
	\[\begin{tikzcd}
	S/(S_{+}^{W_{J}}) & \arrow[l, two heads] S/(S_{+}^{W})\\ S^{W_{I\cap J}}/(S_{+}^{W_{J}}) \arrow[u, hook]& S^{W_{I}}/(S_{+}^{W}) \arrow[l] \arrow[u, hook] 
	\end{tikzcd}\]
	Here the top map is simply projection as $W_{J} \subset W$ implies that $S_{+}^{W} \subset S_{+}^{W_{J}}$. The vertical maps $\theta^{*}$ and $\pi^{*}$ are injective so we can often reduce to the case of full flag varieties, as in \ref{sec-normbun}. %Given the projection formula and the following lemma, it suffices to compute $\iota_{*}(1)$ to describe the whole pushforward map $\iota_{*}$.
	%\begin{lemma}
	%	The pullback map $\iota^{*}$ is surjective.
	%\end{lemma}
	
	%\begin{proof}
	%	$L_{J}/L_{J}\cap P_{I} \simeq P_{J}/P_{I} = \overline{G_{w_{0,J}}}/P_{I} = X_{w_{0,J}}$	is a Schubert variety in $G/P_{I}$
	%	where $w_{0,J} \in W_{J} \subset W$ is the longest element of $W_{J}$ and $G_{w} = B\dot{w}P \subset G$ with closure given by $\overline{G_{w}} = \sqcup_{v \leq w}G_{v}$ \cite[13.3 (4)]{Jantzen}.
		
	%	The Bruhat stratification on $L/L\cap P$ is a refinement of the Bruhat stratification on $G/P$ so the pullback map is surjective on cohomology (at least in characteristic zero, see \cite[2.1]{GasharovReiner}) and thus on Chow rings.
	%\end{proof}
	
	\subsection{Normal bundle}
	\label{sec-normbun}
	In order to utilise the self-intersection formula we require some knowledge of the Chern classes of the normal bundle.
	\begin{proposition}
		\label{cor-top-chern-class}Let $\mathcal{N}$ denote the normal bundle of the closed embedding $\iota \colon L_{J}/L_{J}\cap P_{I} \hookrightarrow G/P_{I}$. Then its Chern polynomial is
		$$c_{t}(\theta^{*}\mathcal{N}) = \prod_{\alpha \in \Phi^{+} \setminus (\Phi_{I}^{+} \cup \Phi_{J}^{+})} (1 + \left[\alpha\right] t)$$
	\end{proposition}
	\begin{proof}
		Firstly, \[\begin{tikzcd}
			0 \arrow[r]& \mathcal{T}_{L/L\cap P_{I}} \arrow[r]& \iota^{*}\mathcal{T}_{G/P_{I}} \arrow[r]& \mathcal{N} \arrow[r] & 0
		\end{tikzcd}\] so $c_{t}(\theta^{*}\iota^{*}\mathcal{T}_{G/P_{I}}) = c_{t}(\theta^{*}\mathcal{N})c_{t}(\theta^{*}\mathcal{T}_{L_{J}/L_{J}\cap P_{I}})$.
		
		\vspace{3mm}
		
		Secondly, the tangent bundle is $\mathcal{T}_{G/P_{I}} = \mathcal{V}(\rho)$ where $\rho$ is the composition $$ P_{I} \hookrightarrow G \xrightarrow{\ad} GL(\mathfrak{g}) \rightarrow GL(\mathfrak{g}/\mathfrak{p})$$ with the third map being induced by projection. This is because there is an equivalence of categories $\Vect_{G}(G/P) \simeq \Rep(P)$ given by $E \mapsto (P \rightarrow GL(E_{x}))$ where $x = eP \in G/P$. The tangent bundle is $G$-equivariant. Taking tangent spaces of the projection $G \rightarrow G/P$ we get a surjective $P$-equivariant linear map $\mathfrak{g} = T_{e}(G) \rightarrow T_{x}(G/P) = \mathcal{T}_{G/P,x}$ with kernel $\mathfrak{p}$. Hence, $\mathcal{T}_{G/P,x} \simeq \mathfrak{g}/\mathfrak{p}$. The action of $G$ on $\mathfrak{g}$ is given by the adjoint representation.
		
		\vspace{3mm}
		
		Now, the roots of $\rho$ are $\Phi^{+} \setminus\Phi^{+}_{I}$ $$c_{t}(\pi^{*}\mathcal{T}_{G/P})  = \prod_{\alpha \in \Phi^{+} \setminus \Phi_{I}^{+}} (1 + \left[\alpha\right] t)$$ Similarly, $c_{t}(\theta^{*}\mathcal{T}_{L/L\cap P}) = \prod_{\alpha \in \Phi_{J}^{+} \setminus \Phi_{I \cap J}^{+}} (1 + \left[\alpha\right] t)$.
	\end{proof}

	\section{Siegel case}
	In this section we prove Theorem \ref{th-main-intro}. This confirms Conjecture \ref{conj-bdy} for the Siegel modular variety $\mathcal{A}_{g}$. 
	
	\subsection{Group data}
	
	Let $V$ be a $2g$ dimensional symplectic vector space with symplectic form $\psi$ given by $\begin{pmatrix}
		0 & -J \\ J & 0
	\end{pmatrix}$ 
	where $J = \antidiag(1,\ldots,1)$. Define $G \coloneqq GSp(V,\psi)$. Choose the maximal torus $$T = \diag(zt_{1},\ldots,zt_{g},\bar{z}t_{g}^{-1},\ldots, \bar{z}t_{1}^{-1})$$ and a basis $e_{i}\colon T \mapsto zt_{i}$ of $X^{*}(T) \cong \mathbb{Z}^{g}$. The root system $\Phi(G,T)$ is $C_{g}$. The choice of Borel $B \subset G$ consisting of lower triangular matrices gives simple roots $\Delta = \{e_{1}-e_{2},\ldots,e_{g-1}-e_{g},2e_{g}\}$. Let $I = \Delta \setminus \{2e_{g}\}$ and $J = \Delta \setminus \{e_{1}-e_{2}\}$.
	This gives 
	$$\Phi_{I}^{+} = \{e_{i}- e_{j} \mid 1\leq i \neq j \leq g\}$$
	$$\Phi_{J}^{+} = \{e_{i}\pm e_{j} \mid 2\leq i \leq j \leq g\}$$
	$$\Phi^{+}\setminus (\Phi_{I}^{+}\cap \Phi_{J}^{+}) = \{e_{1}+ e_{i} \mid 1\leq i \leq g\}$$ and a commuting diagram of sub flag varieties:
	\[\begin{tikzcd}
		L_{J}/B \cap L_{J} \arrow[r, hook, "\nu"] \arrow[d, "\theta"]& G/B \arrow[d, "\pi"]\\ L_{J}/P_{I} \cap L_{J} \arrow[r, hook, "\iota"]& G/P_{J}
	\end{tikzcd}\] 
	
	\subsection{Weyl groups}
	The Weyl group can be described as $$W = \{\sigma \in S_{2g} \mid \sigma(i) + \sigma(2g+1-i) = 2g+1 \text{ for } i=1,\ldots,g\} \cong S_{g}\rtimes\{\pm 1\}^{g}$$ The simple reflections are $s_{e_{i}-e_{i+1}} = (i,i+1)(2g-i,2g+1-i)$ and $s_{2e_{g}} = (g,g+1)$.
	$$W_{I} = \langle s_{e_{1}-e_{2}},\ldots, s_{e_{g-1}-e_{g}}\rangle \cong S_{g} \subset W$$ 

\subsection{Chow rings \& pullback map}
The picture for the Chow rings is the following:
\[\begin{tikzcd}
	S/(S_{+}^{W_{J}})\cong \A(L_{J}/B \cap L_{J}) & \A(G/B)\cong S/(S_{+}^{W}) \arrow[l, twoheadrightarrow, "\nu^{*}"]\\ S^{W_{I\cap J}}/(S_{+}^{W_{J}})\cong \A(L_{J}/P_{I} \cap L_{J})  \arrow[u, hook, "\theta^{*}"]& \A(G/P_{J})\cong S^{W_{I}}/(S_{+}^{W}) \arrow[l, "\iota^{*}"] \arrow[u, hook, "\pi^{*}"]
\end{tikzcd}\] 
We describe these rings and the pullback maps between them explicitly in terms of generators and relations. Denote $\lambda_{i} \coloneqq \sigma_{i}(e_{1},\ldots ,e_{g})$ (resp. $\widetilde{\lambda}_{i} \coloneqq \sigma_{i}(e_{1},\ldots,e_{g-1})$) the $i$-th elementary symmetric polynomial in $\{e_{1},\ldots,e_{g}\}$ (resp. $\{e_{1},\ldots,e_{g-1}\}$).
\begin{lemma}
	\label{le-pullback-map}
	 We have the following.$$S^{W_{I}}/(S^{W}_{+}) \cong \mathbb{Q}[\lambda_{1},\ldots, \lambda_{g}]/((1+\lambda_{1}+\ldots+\lambda_{g})(1-\lambda_{1}+\ldots+(-1)^{g}\lambda_{g})-1)$$
		and  $$S^{W_{I\cap J}}/(S^{W_{J}}_{+}) \cong \mathbb{Q}[\widetilde{\lambda}_{1},\ldots, \widetilde{\lambda}_{g-1}]/((1+\widetilde{\lambda}_{1}+\ldots+\widetilde{\lambda}_{g})(1-\widetilde{\lambda}_{1}+\ldots+(-1)^{g-1}\widetilde{\lambda}_{g-1})-1)$$
		and the pullback map $\iota^{*}\colon \A(G/P_{I}) \rightarrow \A(L_{J}/L_{J}\cap P_{I})$ is surjective with kernel $(\lambda_{g})$.
		
\end{lemma}

\begin{proof}
	$W_{I} = \langle s_{e_{1}-e_{2}},\ldots, s_{e_{g-1}-e_{g}}\rangle \cong S_{g}$ implies that $S^{W_{I}} \cong \mathbb{Q}[\lambda_{1},\ldots, \lambda_{g}]$. If $f \in S=\mathbb{Q}[e_{1},\ldots,e_{g}]$ is $W$ invariant then it is invariant under the action of $\{\pm 1\}^{g}$ and thus $f = g(e_{1}^{2},\ldots,e_{g}^{2})$ for some polynomial $g$, so $$\mathbb{Q}[e_{1}.\ldots,e_{g}]^{W}\cong\mathbb{Q}[e_{1}^{2},\ldots,e_{g}^{2}]^{S_{g}}\cong \langle\sigma_{1}(e_{1}^{2},\ldots,e_{g}^{2}),\ldots,\sigma_{g}(e_{1}^{2},\ldots,e_{g}^{2})\rangle$$ by the fundamental theorem of symmetric polynomials.
 
    The required description of $S_{+}^{W}$ follows from Lemma \ref{le-symm-poly} which says that the homogeneous part of $(1+\lambda_{1}+\ldots+\lambda_{g})(1-\lambda_{1}+\ldots+(-1)^{g}\lambda_{g})-1$ of degree $k=2l$ is equal to $(-1)^{l}\sigma_{i}(e_{1}^{2},\ldots,e_{g}^{2})$. The same argument works for $S^{W_{I\cap J}}/(S^{W_{J}}_{+})$.
	
	\vspace{3mm}
	
	Now, we have that $e_{1} \in S^{W_{J}}$, since $s_{e_{i}-e_{i+1}}(e_{1}) = e_{1}$ for $i>2$ and similarly for $s_{2e_{g}}$. Firstly, this implies that $\iota^{*}$ is surjective, since $\iota^{*}\lambda_{i} = \widetilde{\lambda_{i}}$. Moreover,  $\ker(\nu^{*}) = (e_{1})$ and so $\ker(\iota^{*}) = \ker(\nu^{*})\cap S^{W_{I}}/(S_{+}^{W}) = (\lambda_{g})$. 
\end{proof}
We record here for the sake of completeness the following result about symmetric polynomials.
\begin{lemma}
	\label{le-symm-poly}
	 Let $g\geq 1$. For all $1\leq l < g$ the following homogeneous symmetric polynomials in $\mathbb{Q}[x_{1},\ldots,x_{g}]$ of degree $2l$ are equal: $$(-1)^{l}\sigma_{l}(x_{1}^{2},\ldots,x_{g}^{2}) = \sum_{i+j=2l}(-1)^{j}\sigma_{i}(x_{1},\ldots,x_{g})\sigma_{j}(x_{1},\ldots,x_{g})$$
\end{lemma}
\begin{proof}
	Induction on $g$. Base case. Suppose for induction that it holds for all $g_{0}<g$. We make use of the fact that for $1\leq i < g$: $$\sigma_{i}(x_{1},\ldots,x_{g}) = \sigma_{i}(x_{1},\ldots,x_{g-1}) + x_{g}\sigma_{i-1}(x_{1},\ldots,x_{g-1})$$  Indeed, fixing $1\leq l < g$, this yields
	\begin{equation*}
		\begin{aligned}
			& \sum_{i+j=2l}(-1)^{j}\sigma_{i}(x_{1},\ldots,x_{g})\sigma_{j}(x_{1},\ldots,x_{g}) 
			%\\&= 
			%\sum_{i+j=2l}(-1)^{j}(\sigma_{i}(x_{1},\ldots,x_{g-1}) + x_{g}\sigma_{i-1}(x_{1},\ldots,x_{g-1}))(\sigma_{j}(x_{1},\ldots,x_{g-1}) + x_{g}\sigma_{j-1}(x_{1},\ldots,x_{g-1}))
			%\\&= \sum_{i+j=2l}(-1)^{j}\sigma_{i}(x_{1},\ldots,x_{g-1})\sigma_{j}(x_{1},\ldots,x_{g-1}) \\&+ x_{g}\sum_{i+j=2l}(-1)^{j}\sigma_{i-1}(x_{1},\ldots,x_{g-1})\sigma_{j}(x_{1},\ldots,x_{g-1}) \\&+ x_{g}\sum_{i+j=2l}(-1)^{j}\sigma_{i}(x_{1},\ldots,x_{g-1})\sigma_{j-1}(x_{1},\ldots,x_{g-1}) \\&+ x_{g}^{2}\sum_{i+j=2l}(-1)^{j}\sigma_{i-1}(x_{1},\ldots,x_{g-1})\sigma_{j-1}(x_{1},\ldots,x_{g-1}) 
			\\&= \sum_{i+j=2l}(-1)^{j}\sigma_{i}(x_{1},\ldots,x_{g-1})\sigma_{j}(x_{1},\ldots,x_{g-1}) \\&- x_{g}^{2}\sum_{i+j=2l-2}(-1)^{j}\sigma_{i}(x_{1},\ldots,x_{g-1})\sigma_{j}(x_{1},\ldots,x_{g-1})
		\end{aligned}
	\end{equation*}
	Using the inductive hypothesis for $g_{0} = g-1$ we have that 
	\begin{equation*}
		\begin{aligned}
			&\sum_{i+j=2l}(-1)^{j}\sigma_{i}(x_{1},\ldots,x_{g-1})\sigma_{j}(x_{1},\ldots,x_{g-1}) \\&- x_{g}^{2}\sum_{i+j=2l-2}(-1)^{j}\sigma_{i}(x_{1},\ldots,x_{g-1})\sigma_{j}(x_{1},\ldots,x_{g-1}) \\&= (-1)^{l}\sigma_{l}(x_{1}^{2},\ldots,x_{g-1}^{2}) + 
			(-1)^{l}x_{g}^{2}\sigma_{l-1}(x_{1}^{2},\ldots,x_{g-1}^{2}) \\&=
			(-1)^{l}\sigma_{l}(x_{1}^{2},\ldots,x_{g}^{2})
		\end{aligned}
	\end{equation*}
\end{proof}
\subsection{Proof of Theorem \ref{th-main-intro}}
\begin{proposition}
	\label{cor-chern-normal}
	Let $\mathcal{N}$ denote the normal bundle of the embedding $\iota$. The top Chern class vanishes $c_{g}(\mathcal{N}) = 0 \in \A(L_{J}/L_{J}\cap P_{I})$.
\end{proposition}
\begin{proof}
	Firstly, it suffices to prove that $c_{g}(\theta^{*}\mathcal{N}) = 0$ since $\theta^{*}$ is injective. By Proposition \ref{cor-top-chern-class}, $c_{g}(\theta^{*}\mathcal{N}) = \prod_{\alpha \in \Phi^{+} \setminus (\Phi_{I}^{+} \cup \Phi_{J}^{+})}\alpha$.  We have that $e_{1} \in S^{W_{J}}$ since $s_{e_{i}-e_{i+1}}(e_{1}) = e_{1}$ for $i>2$ and similarly for $s_{2e_{g}}$. Moreover, $2e_{1} \in \Phi^{+} \setminus (\Phi_{I}^{+} \cup \Phi_{J}^{+})$. Thus, $c_{g}(\theta^{*}\mathcal{N}) = 0$.
\end{proof}
\begin{theorem}[Theorem \ref{th-main-intro}]
	\label{th-main}
	$[L_{J}/L_{J}\cap P_{I}] = a\lambda_{g} \in \A(G/P_{I})$ for some $a \neq 0$.
\end{theorem}

\begin{proof}
	We have a closed immersion $\iota \colon L_{J}/L_{J}\cap P_{I} \rightarrow G/P_{I}$ and $\iota_{*}(1)$ is proportional to $[L_{J}/L_{J}\cap P_{I}] \in \A(G/P_{I})$. To determine this class up to scalar we apply the pullback map to get $\iota^{*}\iota_{*}(1) = c_{g}\mathcal{N} = 0$ by Corollary \ref{cor-chern-normal}. However, the pullback map is surjective with kernel $\lambda_{g}\A(G/P_{I})$ and so $\iota_{*}(1) = a.\lambda_{g}$ for some $a \in \mathbb{Q}$. Given that $L_{J}/L_{J}\cap P_{I}$ is a closed subvariety in the projective variety $G/P_{I}$, its class in Chow cannot be zero, so $a \neq 0$. 
\end{proof}
What can be said about the constant $a$? For example, what is its sign? A precise statement is the following:
\begin{lemma}
	\label{le-constant}
	Let $\rho_{\Hodge} \in \Rep(P_{I})$ be the representation giving the Hodge vector bundle $\Omega = \mathcal{W}(\rho_{\Hodge})$ on $\mathcal{A}_{g}^{\tor}$. Then $[L_{J}/L_{J}\cap P_{I}] = a_{J,g}c_{g}(\mathcal{V}(\rho_{\Hodge}))$ in $\A(G/P_{I})$ where $(-1)^{g}a_{J,g} > 0$.
\end{lemma}
%We can also try to understand the classes of smaller sub flag varieties. For $1 \leq k \leq g-1$ let $J_{k} = \Delta \setminus \{e_{1}-e_{2},\ldots e_{k}-e_{k+1}\}$. Denote $Y_{g} = G/P_{I}$ and $Y_{g-k} = L_{J_{k}}/L_{J_{k}}\cap P_{I}$.
%\begin{corollary}
%	$[Y_{g-k}] = a_{k}.\lambda_{g}\lambda_{g-1}\ldots \lambda_{g-k+1} \in \A(Y_{g})$ for some $a_{k} \neq 0$. 
%\end{corollary}

%\begin{proof} NEED TO CHECK THIS
%	Do induction on $g$. Suppose it holds for $h \leq g-1$. Let $1 \leq k \leq g-1$. Then in $\A(Y_{g-1})$ we have $$[Y_{g-k}] = [Y_{g-1-(k-1)}] = b_{k-1}\lambda_{g-1}\ldots \lambda_{g-1-(k-1)+1}$$ Pushing forward to $Y_{g}$ we get $$[Y_{g-k}] = b_{k-1}\iota_{*}\iota^{*}(\lambda_{g-1}\ldots \lambda_{g-k+1}) = b_{k-1}\iota_{*}(1).\lambda_{g}\ldots \lambda_{g-k+1}$$ is proportional to $\lambda_{g}\ldots \lambda_{g-k+1}$ up to non-zero scalar.
%\end{proof}
To prove this we compare with the Shimura variety and use positivity statements on both sides together with Hirzebruch-Mumford proportionality. Let $S$ be a Shimura variety of Hodge type. Write $d \coloneqq \dim(S)$. Applying $\ref{subsec-Chowrings}$ to the standard principal bundle on $S^{\tor}$ gives
\[\begin{tikzcd}
	\Rep(P_{I}) \arrow[r, "\mathcal{W}"] \arrow[d, "\mathcal{V}"]& \Vect(S^{\tor}) \\ \Vect(G/P_{I})
\end{tikzcd}\] 
and also for each $1\leq i \leq d$ \[\begin{tikzcd}
	S^{W_{I}} \arrow[r, "c_{i}(\mathcal{W})"] \arrow[d, "c_{i}(\mathcal{V})"]& A^{i}(S^{\tor}) \\ A^{i}(G/P_{I})
\end{tikzcd}\] 
\subsubsection{Hirzebruch-Mumford Proportionality}
	There is some $R \neq 0$ such that for all $f \in \mathbb{Q}[c_{1},\ldots,c_{d}]$ homogeneous of degree $d$ in the graded ring with grading $\deg(c_{i}) = i$, and all $\rho \in \Rep(P_{I})$ there is equality 
	$$\int_{G/P_{I}}f(c_{1}(\mathcal{V}(\rho)),\ldots, c_{d}(\mathcal{V}(\rho))) = R\int_{S^{\tor}}f(c_{1}(\mathcal{W}(\rho)),\ldots, c_{d}(\mathcal{W}(\rho)))$$
\begin{proof}
	See \cite{Mumford} in characteristic zero and \cite[Corollary 7.21]{Wedhorn-Ziegler} in positive characteristic.
\end{proof}

\subsubsection{Proof of Lemma \ref{le-constant}}
    Under the Borel isomorphism $$\A(G/P_{I}) \cong \mathbb{Q}[\lambda_{1},\ldots, \lambda_{g}]/((1+\lambda_{1}+\ldots+\lambda_{g})(1-\lambda_{1}+\ldots+(-1)^{g}\lambda_{g})-1)$$ $c_{g}(\mathcal{V}(\rho_{\Hodge}))$ is proportional to $\lambda_{g}$ so Theorem \ref{th-main-intro} gives that $[L_{J}/L_{J}\cap P_{I}] = a_{J,g}c_{g}(\mathcal{V}(\rho_{\Hodge}))$ for some $a_{J,g} \neq 0$. It remains to show that $(-1)^{g}a_{J,g} > 0$.

    \vspace{3mm}
 
	Firstly, in the Siegel case $S^{\tor} = \mathcal{A}_{g}^{\tor}$ the Hirzebruch-Mumford proportionality constant satisfies: $(-1)^{d}R > 0$. The line bundle $\mathcal{L}\coloneqq \mathcal{V}(\det(\rho_{\Hodge}))$ on $G/P_{I}$ is anti-ample by \cite[Lemma 3.2.5]{GoldringKoskivirta}, \cite[II (4.4)]{Jantzen} and the Hodge line bundle $\omega \coloneqq \det(\Omega) = \mathcal{W}(\det(\rho_{\Hodge}))$ on $S^{\tor}$ is nef.
    \begin{equation*}
        \begin{aligned}
            \int_{G/P_{I}}c_{1}(\mathcal{L}^{\vee})^{d} &= (-1)^{d}\int_{G/P_{I}}c_{1}(\mathcal{L})^{d}
            \\&= (-1)^{d}R\int_{S^{\tor}}c_{1}(\omega)^{d}
        \end{aligned}
    \end{equation*}
     Thus, $(-1)^{d}R > 0$.

    \vspace{3mm}
 
    Secondly, applying Hirzebruch-Mumford proportionality with $\rho_{\Hodge}$ and homogeneous polynomial $f = c_{g}c_{1}^{d-g}$ gives:
	\begin{equation*}
		\begin{aligned}
			\int_{G/P_{I}}[L_{J}/L_{J}\cap P_{I}]c_{1}(\mathcal{L}^{\vee})^{d-g} &= (-1)^{d-g}a_{J,g}\int_{G/P_{I}}c_{g}(\mathcal{V}(\rho_{\Hodge}))c_{1}(\mathcal{L})^{d-g}
            \\&=
			(-1)^{d-g}a_{J,g}R \int_{S^{\tor}}c_{g}(\Omega)c_{1}(\Omega)^{d-g}
			\\&=
			(-1)^{d-g}a_{J,g}R \int_{S^{\tor}}\frac{1}{N_{g,p}}[V_{0}]c_{1}(\omega)^{d-g}	
		\end{aligned}
	\end{equation*}
	where $N_{g,p}$ is a positive constant depending on $g$ and $p$. The last line follows from \cite[Prop 9.2]{vanderGeer} which says that the cycle class of the $p$-rank $0$ locus $V_{0} \subset \mathcal{A}_{g}^{\tor}$ is $[V_{0}] = N_{g,p}c_{g}(\Omega)$. Hence, $ \int_{S^{\tor}}\frac{1}{N_{g,p}}[V_{0}].c_{1}(\mathcal{W}(\rho))^{d-g} \geq 0$. The line bundle $\mathcal{L}$ on $G/P_{I}$ is anti-ample so $\int_{G/P}[L_{J}/L_{J}\cap P_{I}].c_{1}(\mathcal{L}^{\vee})^{d-g} > 0$. This implies that $$(-1)^{d-g}a_{J,g}R \geq 0$$ Thus, $(-1)^{g}a_{J,g} > 0$. \qed

%\section{Pullback map in general}
%\subsection{Bases}

%Due to the Bruhat decomposition which is an affine stratification the cycle class map $\A(G/P_{I}) \longrightarrow \coH(G/P_{I},\mathbb{Q})$ is an isomorphism. In this section we recall some bases of this graded vector space. The presentation of Borel provides one set of generators for the Chow ring in terms of chern classes of homogenous vector bundles. There is another important set of bases given in terms of the Bruhat decomposition.
%\begin{enumerate}
%	\item $\{[X_{w}]\in A^{d-l(w)}(G/P_{I}) \mid w\in W_{I}\}$
%	\item $\{[X_{w}]\in H_{2l(w)}(G/P_{I}) \mid w\in W_{I}\}$
%	\item $\{[X_{w}]^{PD}\in H^{2d - 2l(w)}(G/P_{I}) \mid w\in W_{I}\}$
%	\item $\{[X_{w}]^{KD}\in H^{2l(w)}(G/P_{I}) \mid w\in W_{I}\}$
%\end{enumerate}
%Note that $[X_{w}]^{KD} = [X_{w_{0}w}]^{PD}$ so this is just a relabelling. The cycle class map sends $[X_{w}]$ to $[X_{w}]^{PD}$. 

%\vspace{3mm}

%The problem of presenting the Schubert classes i.e. basis elements above in terms of Borel's presentation is tackled in... We present a new method which works in very restricted situations.

%Let $G$ be a connected reductive group with root datum $(\phi,\phi^{\vee}...)$ and a choice of positive roots $\phi^{+}$ (equivalently choice of Borel subgroup $B$). Denote $\Delta$ as the simple roots. 

%$[X_{w}] \mapsto [X_{w}]^{PD} = [X_{v}]^{KD} \mapsto [X_{v}]^{KD} = [X_{w_{0,I}v}]^{PD} \mapsto [X_{w_{0,I}v}]$ where $w_{0}v = w$ in ${}^{I}W$.
\section{Other Shimura/flag varieties}

There are other Shimura varieties where we obtain fairly trivial confirmations of Conjecture \ref{conj-bdy}. In the case of the Hilbert modular variety at an unramified prime the parabolic subgroup is the Borel subgroup and so the sub flag variety has dimension $0$. This forces it to be proportional to the only class in degree $0$ in $\A(G/B)\cong \T(S^{\tor})$. In \cite{cooper2021tautological} it is shown that $\T(S) \cong \T(S^{\tor})/(z)$ where $z$ is a generator of the $1$-dimensional space $T_{0}(S^{\tor})$. So in this case, \ref{conj-bdy} is (trivially) confirmed. 

	\bibliographystyle{plain}
	\bibliography{paper2refs}
\end{document}